\documentclass[10pt]{article}
 \makeatletter
\newfont{\footsc}{cmcsc10 at 8truept}
\newfont{\footbf}{cmbx10 at 8truept}
\newfont{\footrm}{cmr10 at 10truept}
\renewcommand{\ps@plain}{%
\renewcommand{\@oddfoot}{\footsc {\footbf }  \footrm\thepage}}
\makeatother  \leftmargin=25mm
\usepackage[centertags]{amsmath}
\usepackage{amsfonts}
\usepackage{amsthm}
\usepackage{newlfont}
\usepackage{graphicx}
\usepackage{epstopdf}
\usepackage{subfigure}
\usepackage[utf8]{inputenc}
\usepackage{indentfirst}
\usepackage{latexsym}
\usepackage{float}
\usepackage{cite}
\usepackage{CJK}
\usepackage{enumitem}
\numberwithin{equation}{section}

\newtheorem{thm}{Theorem}[section]

\newtheorem{lem}[thm]{Lemma}

\theoremstyle{definition}

\theoremstyle{remark}
\newtheorem{rem}[thm]{Remark}

\allowdisplaybreaks[4]

\begin{document}
\title{ Borsuk's partition problem in four-dimensional $\ell_{p}$ space}
\date{}
\author{Jun Wang \footnote{Center for Applied Mathematics, Tianjin University, Tianjin, China 300354. E-mail: kingjunjun@tju.edu.cn. } and  Fei Xue \footnote{School of Mathematical Sciences, Nanjing Normal University, No.1 Wenyuan Road Qixia District, Nanjing, China 210046. E-mail: 05429@njnu.edu.cn.}}

\maketitle

\begin{abstract}In 1933, Borsuk made a conjecture that every $n$-dimensional bounded set can be divided into $n+1$ subsets of smaller diameter. Up to now, the problem is still open for $4\leq n\leq 63$. In this paper, we firstly discuss the Banach-Mazur distance between the $n$-dimensional cube and the $\ell_{p}$ ball $(1\leq p< 2)$, then we study the generalized Borsuk's partition problem in metric spaces and prove that all bounded sets $X$ in every four-dimensional $\ell_{p}$ space can be divided into $2^4$ subsets of smaller diameter.

\end{abstract}

\maketitle

\section{Introduction}
Let $\mathbb{E}^{n}$ be the $n$-dimensional Euclidean space, and in this paper an $n$-dimensional vector $\mathbf{x}\in \mathbb{E}^{n}$ is always treated as a column vector. Let $K$ denote an $n$-dimensional convex body, a bounded compact convex set with non-empty interior $\mathrm{int}(K)$ and boundary $\partial(K)$. By $\mathcal{K}^{n}$ we denote the set of convex bodies in $\mathbb{E}^{n}$.

Let $d(X)$ denote the diameter of a bounded set $X$ of $\mathbb{E}^{n}$ defined by $$d(X)=\sup\{\|\mathbf{x},\mathbf{y}\|:\mathbf{x},\mathbf{y}\in X\},$$
where $\|\mathbf{x},\mathbf{y}\|$ denotes the Euclidean distance between
$\mathbf{x}$ and $\mathbf{y}$. Let $b(X)$ be the smallest number of subsets $X_{1},X_{2},...,X_{b(X)}$ of $X$ such that
$$X=\bigcup_{i=1}^{b(X)}X_{i}$$
and $d(X_{i})< d(X)$ holds for all $i\leq b(X)$. In 1933, K. Borsuk \cite{Borsuk} proposed the following problem:

\leftline{$\mathbf{Borsuk's\ partition\ problem}.$\ \  Is it true that} $$b(X)\leq n+1$$ holds for every bounded set $X$ in $\mathbb{E}^{n}$ ?

Usually, the positive statement of this problem is referred as Borsuk's conjecture. K. Borsuk \cite{Borsuk} proved that the inequality $b(X)\leq 3$ holds for any bounded set $X\subseteq \mathbb{E}^{2}$. For $n=3$, Borsuk's conjecture was confirmed by H. G. Eggleston\cite{Eggleston} in 1955. In 1945, H. Hadwiger\cite{Hadwiger} proved that the inequality $b(K)\leq n+1$ holds for every $n$-dimensional convex body $K$ with smooth boundary. However, in 1993, J. Kahn and G. Kalai\cite{Kahn} discovered counterexamples to Borsuk's conjecture in high dimensions. In 2021, C. Zong \cite{ZongC} gave a computer proof program to deal with this challenging problem. Up to now, the problem is still open for $4\leq n\leq 63$.

Let $\mathbb{M}^n$ denote the Minkowski space with respect to the metric $\|\cdot\|_{C}$ determined by a centrally symmetric convex body $C$. For a bounded set $X\subseteq\mathbb{M}^n$, let $d_{C}(X)$ denote the diameter of $X$ and let $b_{C}(X)$ denote the smallest number such that $X$ can be divided into $b_{C}(X)$ subsets each of which has the diameter strictly smaller than $d_{C}(X)$.

In 1957, B. Gr\"{u}nbaum\cite{Grunbaum} firstly studied the problem in Minkowski planes $\mathbb{M}^{2}$. It was mentioned in \cite{Boltyanski} that for every bounded set $X\subseteq \mathbb{M}^{2}$, if the unit ball $C$ of $\mathbb{M}^{2}$ is a not a parallelogram, then the inequality $b_{C}(X)\leq 3$ holds; otherwise, the inequality $b_{C}(X)\leq 4$ holds.

The covering number $\gamma(K)$ is the smallest number of translates of $\lambda K$ $(0<\lambda< 1)$ such that their union contains $K$. In 1957, H. Hadwiger \cite{Hadwigerr} raised the following conjecture, which has a close relation with the Borsuk's partition problem.

\leftline{$\mathbf{Hadwiger's\ covering\ conjecture.}$ Every convex body $K$ in $\mathbb{E}^{n}$ can be covered by $2^{n}$ translates of $\lambda K$}
\leftline{(or $\mathrm{int}(K)$), where $\lambda$ is a suitable positive number satisfying $\lambda<1$.}

The two-dimensional case had been solved by F. W. Levi\cite{Levi}. However, this conjecture is open for all $n\geq3$ untill now. In 2010, C. Zong\cite{Zong} proposed a four-step program to attack this conjecture.

In 1965, V. G. Boltyanski and I. T. Gohberg \cite{Boltyanskii} proved that
\begin{equation}\label{eq0.1}
 b_{C}(X)\leq \gamma(\widehat{X})
\end{equation}
holds for all metrics and all bounded sets $X$ in $\mathbb{E}^{n}$, where $\widehat{X}$ denotes the closed convex hull of $X$. Based on this fact, they also proposed the following problem:

\leftline{$\mathbf{Problem\ 1}$ Is it true that}
$$b_{C}(X)\leq 2^n$$
holds for all $n$-dimensional metric spaces $\mathbb{M}^n$ and all bounded sets $X\subset \mathbb{M}^{n}$?


In this paper, we concern the $\ell_{p}$ space. For a real number $p\geq1$, the $p$-norm of a $\mathbf{x}\in\mathbb{E}^{n}$ is defined by
 $$\|\mathbf{x} \|_{p}=(|x_{1}|^{p}+|x_{2}|^{p}+\cdots+|x_{n}|^{p})^{\frac{1}{p}}.$$
The maximum norm is the limit of the $p$-norm for $p\rightarrow\infty$, a.k.a.,
$$\|\mathbf{x} \|_{\infty}=\max\{|x_{1}|,|x_{2}|,\cdots,|x_{n}|\}.$$
Let $\mathbb{M}_{p}^{n}$ denote the $n$-dimensional $\ell_{p}$ space, the $n$-dimensional linear space with the $p$-norm.
The $n$-dimensional $\ell_{p}$ ball is denoted by
$$C_{n,p}=\{\mathbf{x}\in \mathbb{R}^{n}:\|\mathbf{x}\|_{p}\leq1\}.$$
Denote by
$$C_{n}=\{\mathbf{x}\in\mathbb{R}^{n}:\|\mathbf{x} \|_{\infty}\leq1\}=[-1,1]^{n},$$
the $n$-dimensional unit cube and denote the vertices of $C_{n}$ by $\{-1,1\}^{n}$.

In 2009, L. Yu and C. Zong\cite{Yu} studied Problem 1 and obtained that $b_{C_{3,p}}(X)\leq 2^3$ holds for all bounded sets $X$ in every three-dimensional $\ell_{p}$ space. In 2021, Y. Lian and S. Wu\cite{Lian} showed that each set $X$ having diameter $1$ in three-dimensional $\ell_{p}$ space can be represented as the union of $2^3$ subsets of $X$ whose diameters are at most $0.925$.
In this paper, we continue studing the above problem in $\mathbb{M}_{p}^{4}$. Our main result is:
\begin{thm}\label{thm1.1}
 In every four-dimensional $\ell_{p}$ space
 $$b_{C_{4,p}}(X)\leq 2^4$$
 holds for all bounded sets $X$.
\end{thm}
In order to prove this theorem, we rely on the Banach-Mazur distance. The Banach-Mazur distance between two $\mathbf{o}$-symmetric convex bodies $K$ and $L$ is defined as
$$d_{BM}(K,L)=\min\{r>0:K\subset gL\subset rK,g\in GL(n,\mathbb{R})\},$$
where $ GL(n,\mathbb{R})$ is the set of invertible linear operators.

In \cite{Xue}, F. Xue proved that
$$ \alpha\sqrt{n}\leq d_{BM}(C_{n},C_{n,1})\leq (\sqrt{2}+1)\sqrt{n}$$
and $\alpha$ can actually be improved to $\frac{1}{\sqrt{2}}$. Here, we generalize the results to $C_{n,p}$ as follows:
\begin{thm}\label{thm1.2}
$$ 2^{\frac{1}{2}-\frac{1}{p}}\sqrt{n}\leq d_{BM}(C_{n},C_{n,p})\leq (\sqrt{2}+1)\sqrt{n}$$ for $1\leq p<2$.
\end{thm}
This paper is organized as follows: In Section 2, we study the Banach-Mazur distance between $C_{n}$ and $C_{n,p}$ $(1\leq p\leq 2)$; In Section 3, we prove Theorem \ref{thm1.1}.

\section{On the Banach-Mazur distance between the cube and the $\ell_{p}$ ball}

In order to find the Banach-Mazur distance between $C_{n}$ and $C_{n,p}$ $(1\leq p\leq 2)$ , one needs to find the minimum $r>0$ such that there exists $g\in GL(n,\mathbb{R})$ with
\begin{equation}\label{eq0}
  C_{n,p}\subset gC_{n}\subset rC_{n,p}.
\end{equation}
Assume that $g$ is the linear transformation with row vectors $\mathbf{c}_{1},\mathbf{c}_{2},...,\mathbf{c}_{n}$, i.e. $g=(\mathbf{c}_{i})_{n\times 1}$.
Since $C_{n}=\mathrm{conv}\{\Sigma_{i=1}^{n}\sigma_{i}\mathbf{e}_{i}:\sigma_{i}\in\{-1,1\},i=1,...,n\}=\mathrm{conv}\{\mathbf{v}:\mathbf{v}=(\pm1,\pm1,...,\pm1)^{T}\}$ , where $\mathbf{e}_{i}$ is the $i$-th unit vector, we have
 $$ gC_{n}=\mathrm{conv}\mathrm\{g\mathbf{v},\mathbf{v}\in\{-1,1\}^{n} \}.$$
The maximum distance between any vertex of $gC_{n}$ and $\mathbf{o}$ is
$$\mathop{\max}\limits_{\mathbf{v}\in\{-1,1\}^{n}}\|g\mathbf{v}\|_{p}=\mathop{\max}\limits_{\mathbf{v}\in\{-1,1\}^{n}}\|(\langle\mathbf{c}_{1},\mathbf{v}\rangle, \langle\mathbf{c}_{2},\mathbf{v}\rangle,...,\langle\mathbf{c}_{n},\mathbf{v}\rangle)^{T}\|_{p}.$$
Here we have
\begin{equation}\label{eq1}
 gC_{n}\subset\left(\mathop{\max}\limits_{\mathbf{v}\in\{-1,1\}^{n}}\|g\mathbf{v}\|_{p}\right)C_{n,p}.
\end{equation}

From the left part in \eqref{eq0}, we have
 \begin{equation*}
g^{-1}C_{n,1}= g^{-1}C_{n}^{\ast}=( gC_{n})^{\ast}\subset C_{n,p}^{\ast}=C_{n,q},\  where\  \frac{1}{p}+\frac{1}{q}=1.
 \end{equation*}
 Then
 \begin{equation}\label{eq2}
 \mathop{\max}\limits_{i}\|g^{-1}\mathbf{e}_{i}\|_{q}\leq 1.
 \end{equation}
Combining \eqref{eq2} with \eqref{eq1}, we have
\begin{equation}\label{eq3}
d_{BM}(C_{n},C_{n,p})=r=\mathop{\min}\limits_{g }\mathop{\max}\limits_{\mathbf{v}\in\{-1,1\}^{n}}\|g\mathbf{v}\|_{p},
\end{equation}
where $g\in GL(n,\mathbb{R})$ and $\mathop{\max}\limits_{i}\|g^{-1}\mathbf{e}_{i}\|_{q}\leq 1$.

In order to show the upper bound of the distance, we use the Hadamard matrix. A Hadamard matrix is a square matrix whose entries are either $+1$ or $-1$, and whose rows are mutually orthogonal.
Sylvester \cite{SYLVESTER} provided one way to construct Hadamard matrices. Let
$$H_{1}=(1)$$
$$H_{2}=\begin{pmatrix}
1&1\\
1&-1
\end{pmatrix}$$

$$H_{2^{k}}=\begin{pmatrix}
H_{2^{k-1}}&H_{2^{k-1}}\\
H_{2^{k-1}}&-H_{2^{k-1}}
\end{pmatrix}$$
for $k\geq2$, then $H_{2^{k}}$ are all Hadamard matrices.

The Hadamard conjecture proposes that a Hadamard matrix of order $4k$ exists for every positive integer $k$. So far this conjecture is still open.
\begin{lem}\label{lem1}
 In dimension $n=2^{k}$, $k\in \mathbb{Z}^{+}$, we have $d_{BM}(C_{n},C_{n,p})\leq \sqrt{n}$ for $1\leq p\leq 2$.
\end{lem}
\begin{proof}
 In dimension $n=2^{k}$, there exists a Hadamard matrix $H_{n}$ with row vectors $\mathbf{r}_{1},\mathbf{r}_{2},...,\mathbf{r}_{n}$. Choose the matrix $g=\frac{1}{n^{\frac{1}{p}}}H_{n}$ , and $H_{n}^{-1}=\frac{1}{n}H_{n}^{T}$. Then we have
$$\mathop{\max}\limits_{i}\|g^{-1}\mathbf{e}_{i}\|_{q}=\mathop{\max}\limits_{i}\|n^{\frac{1}{p}}\cdot\frac{1}{n}H_{n}^{T}\mathbf{e}_{i}\|_{q}=n^{\frac{1}{p}}\cdot\frac{1}{n}\cdot n^{\frac{1}{q}}=1.$$
Thus, we have \begin{align*}
               r &= \mathop{\max}\limits_{\mathbf{v}\in\{-1,1\}^{n}}\|g\mathbf{v}\|_{p}\\ &=\frac{1}{n^{\frac{1}{p}}}\mathop{\max}\limits_{\mathbf{v}\in\{-1,1\}^{n}}\|H_{n}\mathbf{v}\|_{p}\\
                 &=\frac{1}{n^{\frac{1}{p}}}\mathop{\max}\limits_{\mathbf{v}\in\{-1,1\}^{n}}\|(\langle\mathbf{r}_{1},\mathbf{v}\rangle, \langle\mathbf{r}_{2},\mathbf{v}\rangle,...,\langle\mathbf{r}_{n},\mathbf{v}\rangle)^{T}\|_{p} \\
                 & =\frac{1}{n^{\frac{1}{p}}}\mathop{\max}\limits_{\mathbf{v}\in\{-1,1\}^{n}}(|\langle\mathbf{r}_{1},\mathbf{v}\rangle|^{p}+ |\langle\mathbf{r}_{2},\mathbf{v}\rangle|^{p}+...+|\langle\mathbf{r}_{n},\mathbf{v}\rangle|^{p})^{\frac{1}{p}}\\
                  & \leq \frac{1}{n^{\frac{1}{p}}}\mathop{\max}\limits_{\mathbf{v}\in\{-1,1\}^{n}}\left((\langle\mathbf{r}_{1},\mathbf{v}\rangle^{2}+ \langle\mathbf{r}_{2},\mathbf{v}\rangle^{2}+...+\langle\mathbf{r}_{n},\mathbf{v}\rangle^{2})^{\frac{p}{2}}\cdot n^{\frac{2-p}{2}}\right)^{\frac{1}{p}} \ \ \ (H\ddot{o}lder \ \ Inequality) \\
                  & =\frac{1}{n^{\frac{1}{p}}}\mathop{\max}\limits_{\mathbf{v}\in\{-1,1\}^{n}}\left((n\cdot\|\mathbf{v}\|_{2}^{2})^{\frac{p}{2}}\cdot n^{\frac{2-p}{2}}\right)^{\frac{1}{p}}\\
                  & =\frac{1}{n^{\frac{1}{p}}}(n^{p}\cdot n^{{\frac{2-p}{2}}})^{\frac{1}{p}}\\
                  & =\sqrt{n}
              \end{align*}
\end{proof}
\begin{rem}\label{rem1}
By Lemma \ref{lem1}, let $\sigma(p)$ denote the smallest number such that there exists a parallelotope $P=gC_{4}$ satisfying
 $P\subseteq C_{4,p}\subseteq \sigma(p)P$. When we choose $g$ to be $\frac{1}{4^{\frac{1}{p}+\frac{1}{2}}}H_{4}$, where $H_{4}$ is a Hadamard matrix, then we have $\sigma(p)\leq 2$.
\end{rem}
{\noindent\bfseries Proof of the upper bound in Theorem \ref{thm1.2}.}
By induction, assume that in dimension $t\leq 2^{k}$ the upper bound of the distance is not bigger than $(\sqrt{2}+1)\sqrt{t}$ with cube determined by $g_{t}$. That is to say, we have $$C_{t,p}\subset g_{t}C_{t}\subset rC_{t,p}$$ and $$r=\mathop{\max}\limits_{\mathbf{v}\in\{-1,1\}^{t}}\|g_{t}\mathbf{v}\|_{p}\leq (\sqrt{2}+1)\sqrt{t}$$
with $$\mathop{\max}\limits_{1\leq i\leq t}\|g_{t}^{-1}\mathbf{e}_{i}\|_{q}\leq 1.$$

Then in dimension $n=2^{k}+t$ where $t\leq 2^{k}$, let
$$g_{2^{k}+t}=\begin{pmatrix}
\frac{1}{(2^{k})^{\frac{1}{p}}}H_{2^{k}}&0\\
0&g_{t}
\end{pmatrix}
.$$
Let $e=\mathop{\max}\limits_{1\leq i\leq 2^{k} }\left\|\left(\frac{1}{(2^{k})^{\frac{1}{p}}}H_{2^{k}}\right)^{-1}\mathbf{e}_{i}\right\|_{q}$, $f=\mathop{\max}\limits_{1\leq i\leq t}\|(g_{t})^{-1}\mathbf{e}_{i}\|_{q}$, obviously, we have $f=e=1$ and
$$\mathop{\max}\limits_{i}\left\|\left(g_{2^{k}+t}\right)^{-1}\mathbf{e}_{i}\right\|_{q}=\max\{e,f\}=1.$$
By Lemma \ref{lem1}, the distance is therefore
\begin{align*}
  \mathop{\max}\limits_{\mathbf{v}\in\{-1,1\}^{n}}\|g_{2^{k}+t}\mathbf{v}\|_{p}
  & \leq\mathop{\max}\limits_{\mathbf{v}\in\{-1,1\}^{2^{k}}}\left\|\frac{1}{(2^{k})^{\frac{1}{p}}}H_{2^{k}}\mathbf{v}\right\|_{p}+\mathop{\max}\limits_{\mathbf{v}\in\{-1,1\}^{t}}\|g_{t}\mathbf{v}\|_{p}\\
  & \leq \sqrt{(2^{k})}+\mathop{\max}\limits_{\mathbf{v}\in\{-1,1\}^{t}}\|g_{t}\mathbf{v}\|_{p}\\
  & = \sqrt{(2^{k})}+(\sqrt{2}+1)\sqrt{t}\\
  & \leq (\sqrt{2}+1)\sqrt{2^{k}+t}\\
  & =(\sqrt{2}+1)\sqrt{n}
\end{align*}
The proof for the upper bound is finished.\qed

The Hadamard conjecture predicts the existence of a Hadamard matrix in dimension $n=4k$. When there indeed exists a Hadamard matrix $H_{n}$, $n=4k$. The distance between $C_{n,p}$ and the cube determined by $H_{n}$ will not be larger than $\sqrt{n}$ for $1\leq p\leq 2$.

When $n=4k+j$, $j<4$, let the cube be determined by
$$g_{4k+j}=\begin{pmatrix}
I_{j}&0\\
0&\frac{1}{({4k})^{\frac{1}{p}}}H_{4k},
\end{pmatrix}$$
then $$\mathop{\max}\limits_{i}\|\left(g_{4k+j})^{-1}\mathbf{e}_{i}\right\|_{q}=1,$$
so the distance is
\begin{align*}
                r &= \mathop{\max}\limits_{\mathbf{v}\in\{-1,1\}^{n}}\|g_{4k+j}\mathbf{v}\|_{p}\\
                 & \leq\mathop{\max}\limits_{\mathbf{v}\in\{-1,1\}^{4k}}\left\|\frac{1}{({4k})^{\frac{1}{p}}}H_{4k}\mathbf{v}\right\|_{p}+\mathop{\max}\limits_{\mathbf{v}\in\{-1,1\}^{j}}\|I_{j}\mathbf{v}\|_{p}\\
                 & \leq\sqrt{4k}+j^{\frac{1}{p}}\\
                  & <\sqrt{n}+3.
              \end{align*}
Therefore the upper bound will be $\sqrt{n}+3$ for all $n$ if the Hadamard conjecture is ture.

{\noindent\bfseries Proof of the lower bound in Theorem \ref{thm1.2}.}
Recall that
 $$d_{BM}(C_{n},C_{n,p})=\mathop{\min}\limits_{g}\mathop{\max}\limits_{\mathbf{v}\in\{-1,1\}^{n}}\|g\mathbf{v}\|_{p},$$
where $g\in GL(n,\mathbb{R})$ and $\mathop{\max}\limits_{i}\|g^{-1}\mathbf{e}_{i}\|_{q}\leq 1$. Without loss of generality, we consider $\mathrm{det}(g)>0$ and write $g=\mathrm{det}(g)^{\frac{1}{n}}N$, where $N\in SL(n,\mathbb{R})$.

By using the Power Mean Inequality and $\|g^{-1}\mathbf{e}_{i}\|_{q}\leq 1$ for $i=1,...,n$ and $q\geq 2$, we have
 \begin{equation}\label{eq4}
\mathrm{det}(g^{-1})\leq \prod_{i=1}^{n}\|g^{-1}\mathbf{e}_{i}\|_{2}\leq n^{\frac{n}{2}}\prod_{i=1}^{n}\left(\frac{\|g^{-1}\mathbf{e}_{i}\|_{q}^{q}}{n}\right)^{\frac{1}{q}}=n^{\frac{n}{2}-\frac{n}{q}}.
 \end{equation}
Let the row vectors of $N$ be $N_{j}$, i.e. $N=(N_{j})_{n\times1}$ , then we have
$$\|N\mathbf{v}\|_{p}=(|\langle N_{1},\mathbf{v}\rangle|^{p}+...+|\langle N_{n},\mathbf{v}\rangle|^{p})^{\frac{1}{p}}.$$
Also, since $\mathrm{det}(N)=1$, by the definition of determinant,
$$\prod_{i=1}^{n}\|N_{j}\|_{2}\geq1,$$
and by the Arithmetic-geometric Inequality,
\begin{equation}\label{eq5}
\left(\frac{\sum_{j=1}^{n}\|N_{j}\|_{2}^{p}}{n}\right)^{\frac{1}{p}}\geq1.
 \end{equation}
By the  Khintchine Inequality,
$$\left(\frac{1}{2^{n}}\mathop{\sum}\limits_{\mathbf{v}\in\{-1,1\}^{n}}|\langle \mathbf{u},\mathbf{v}\rangle|^{p}\right)^{\frac{1}{p}}\geq 2^{\frac{1}{2}-\frac{1}{p}}\|\mathbf{u}\|_{2},$$
 and
  \begin{equation}\label{eq6}
  \frac{1}{2^{n}}\mathop{\sum}\limits_{\mathbf{v}\in\{-1,1\}^{n}}|\langle \mathbf{u},\mathbf{v}\rangle|^{p}\geq 2^{\frac{p}{2}-1}\|\mathbf{u}\|_{2}^{p},
  \end{equation}
Thus,
\begin{align*}
                \mathop{\max}\limits_{\mathbf{v}\in\{-1,1\}^{n}}\|g\mathbf{v}\|_{p} &= \mathrm{det}(g)^{\frac{1}{n}}\mathop{\max}\limits_{\mathbf{v}\in\{-1,1\}^{n}}\|N\mathbf{v}\|_{p}\\
                 & =\mathrm{det}(g)^{\frac{1}{n}}\left(\mathop{\max}\limits_{\mathbf{v}\in\{-1,1\}^{n}}\sum_{j=1}^{n}|\langle N_{j},\mathbf{v}\rangle|^{p}\right)^{\frac{1}{p}}\\
                 & \geq \mathrm{det}(g)^{\frac{1}{n}}\left(\frac{1}{2^{n}}\mathop{\sum}\limits_{\mathbf{v}\in\{-1,1\}^{n}}\sum_{j=1}^{n}|\langle N_{j},\mathbf{v}\rangle|^{p}\right)^{\frac{1}{p}}\\
                  & =\mathrm{det}(g)^{\frac{1}{n}}\left(\frac{1}{2^{n}}\sum_{j=1}^{n}\mathop{\sum}\limits_{\mathbf{v}\in\{-1,1\}^{n}}|\langle N_{j},\mathbf{v}\rangle|^{p}\right)^{\frac{1}{p}}\\
                  &\geq \mathrm{det}(g)^{\frac{1}{n}}\left(2^{\frac{p}{2}-1}\sum_{j=1}^{n}\|N_{j}\|_{2}^{p}\right)^{\frac{1}{p}}\ \ \ \ \ \ \ \  \eqref{eq6}\\
                   &\geq  \mathrm{det}(g)^{\frac{1}{n}}\left(2^{\frac{p}{2}-1}\cdot n\right)^{\frac{1}{p}}\ \ \ \ \ \ \ \ \ \ \ \ \ \  \eqref{eq5}\\
                    &\geq \frac{1}{n^{\frac{1}{p}-\frac{1}{2}}}\cdot n^{\frac{1}{p}}\cdot 2^{\frac{1}{2}-\frac{1}{p}}\ \ \ \ \ \ \ \ \ \ \ \ \ \  \ \ \eqref{eq4}\\
                    &=2^{\frac{1}{2}-\frac{1}{p}}\cdot\sqrt{n}.
              \end{align*}
The proof for the lower bound is finished.\qed

\section{Borsuk's partition problem in four-dimensional $\ell_{p}$ space}
In order to prove Theorem \ref{thm1.1}, let us consider three main situations.

\subsection{$p>2$}

If $p>2$, there exists a parallelotope $P=C_{4}$ satisfying
             \begin{equation}
             \frac{1}{4^{\frac{1}{p}}}P\subseteq C_{4,p}\subseteq P.\label{eq4.1}
             \end{equation}
    For each bounded set $X$, there is a point $\mathbf{x}$ and a minimal number $\tau$ such that
    $$X+\mathbf{x}\subseteq \tau P.$$
   Since
    $$d_{C_{4,p}}(X+\mathbf{x})=d_{C_{4,p}}(X),$$
    and $X+\mathbf{x}$ touches at least one pair of opposite facets of $\tau P$, we have
    \begin{equation}
    d_{C_{4,p}}(X)\geq 2\tau.\label{eq4.2}
    \end{equation}
    On the other hand, $\tau P$ can be covered by exactly $16$ translates of $\frac{\tau}{2}P$. By \eqref{eq4.1} it follows that $\tau P$ can be covered by $16$ translates of $\frac{\tau 4^{\frac{1}{p}}}{2}C_{4,p}$ and $X$ can be divided into $16$ corresponding subsets $X_{1},X_{2},...,X_{16}$ with
    \begin{equation}
    d_{C_{4,p}}(X_{i})\leq d_{C_{4,p}}\left(\frac{\tau 4^{\frac{1}{p}}}{2}C_{4,p}\right)=4^{\frac{1}{p}}\tau<2\tau \label{eq4.3}
    \end{equation}
    for $i=1,...,16$. Therefore, \eqref{eq4.2} and \eqref{eq4.3} together yield $$b_{C_{4,p}}(X)\leq 2^4.$$

\begin{rem}
Using the same method, we can prove $b_{C_{n,p}}(X)\leq 2^n$ holds for all $\log_{2}n < p\leq+\infty,\ n\geq3$ and all bounded set $X$.
\end{rem}

\subsection{$1< p \leq 2$}
If $1< p \leq 2$, we recall from Lemma \ref{lem1} and Remark \ref{rem1} that
\begin{equation}\label{eq3.4}
 \frac{1}{2}Q\subseteq C_{4,p}\subseteq Q,
\end{equation}
where  $Q=gC_{4}$ and
$$g=\frac{1}{4^{\frac{1}{p}}}\begin{pmatrix}
1&-1&1&1\\
1&1&-1&1\\
1&1&1&-1\\
-1&1&1&1
\end{pmatrix}.
$$

Denote by $w(X,\mathbf{u})$ the Euclidean width of $X$ in the direction $\mathbf{u}$. Let $\mathbf{u}_{i}=g\mathbf{e}_{i}$, i.e. $\mathbf{u}_{1}=4^{-\frac{1}{p}}(1,1,1,-1)^{T}$, then $w(C_{4,p},\mathbf{u}_{i})=4^{1-\frac{1}{p}}$ for $i=1,...,4$.

For each bounded set $X$ with $d_{C_{4,p}}(X)=2$ in $M_{p}^{4}$, we have $w(X,\mathbf{u}_{i})\leq4^{1-\frac{1}{p}}$ for $i=1,...,4$ and the equality holds if and only if there exists $\mathbf{a}_{i},\mathbf{b}_{i}\in X$ such that
 \begin{equation}\label{eq3.5}
\mathbf{a}_{i}-\mathbf{b}_{i}=2 \mathbf{u}_{i}.
 \end{equation}

Up to translation, we may assume that $X\subseteq \cap_{i\in[4]}\{\mathbf{x}:|\langle \mathbf{x}, \mathbf{u}_{i}\rangle|\leq w_{i}\}=Q_{X}$ with $w_{i}\leq 4^{1-\frac{2}{p}}$. In fact, $Q=\cap_{i\in[4]}\{\mathbf{x}:|\langle \mathbf{x}, \mathbf{u}_{i}\rangle|\leq 4^{1-\frac{2}{p}}\}$.
Now we consider two cases:
\begin{enumerate}
  \item If there exists some $w_{i}<4^{1-\frac{2}{p}}$, then we have $X\subseteq Q_{X} \mathop{\subset}\limits_{\neq} Q$ and $d_{p}(Q_{X})<4$. In this case, one can divided $ Q_{X}$ into 16 smaller copies of $\frac{1}{2} Q_{X}$ with $d_{C_{4,p}}(\frac{1}{2}Q_{X})<2$. Then $X$ can also be divided into $16$ corresponding parts with diameter strictly smaller than $2$. Thus, $b_{C_{4,p}}(X)\leq 2^{4}$.
  \item If $w_{i}=4^{1-\frac{2}{p}}$ for all $i=1,..,4$, let $F_{i}=\{\mathbf{x}:\langle \mathbf{x}, \mathbf{u}_{i}\rangle= 4^{1-\frac{2}{p}}\}$ and $F_{-i}=\{\mathbf{x}:\langle \mathbf{x}, -\mathbf{u}_{i}\rangle= 4^{1-\frac{2}{p}}\}$, then $X$ touches each pair of opposite facets of $Q$. Assuming that $X$ touches $Q \cap F_{i}$ at one point $\mathbf{a}_{i}$ and touches $Q \cap F_{i}$ at point $\mathbf{b}_{i}$ satisfying \eqref{eq3.5}. In addition, since $C_{4,p}$ is strictly convex when $1<p\leq2$, then $X$ cannot touch $F_{i}$ (as well as $F_{-i}$ ) at more than one point.  Also, all $\mathbf{a}_{i}$, $\mathbf{b}_{i}$ must be in the interior of each facet of $Q$. If not, suppose $\mathbf{a}_{1}$ is on the boundary of one facet of $Q$. Without of loss generality, let $\mathbf{a}_{1}\in Q\cap F_{1}\cap F_{2}$, by $\mathbf{a}_{1}-\mathbf{b}_{1}=2 \mathbf{u}_{1}$, then $\mathbf{b}_{1}\in (Q\cap F_{-1}\cap F_{2})$. Since $d_{C_{4,p}}(X)=2$ and $\mathbf{a}_{1},\mathbf{b}_{1}\in (Q\cap F_{2})$, there is no point of $X$ on the opposite facet $(Q\cap F_{-2})$, which contradicts to the assumption that $X$ intersects all facets of $Q$.

 For $1<p\leq 2$, by the strictly convexity of $C_{4,p}$ and \eqref{eq3.4}, the diameter of $\frac{1}{2}Q$ in $M_{p}^{4}$ is only determined by its eight pairs of symmetric vertices: $$d_{C_{4,p}}\left(\frac{1}{2}Q\right)=2=d_{C_{4,p}}\left(\frac{1}{2}g\mathbf{v},\frac{1}{2}g(-\mathbf{v})\right),\ \ \mathbf{v}\in\{1,-1\}^{4}=\Sigma_{i=1}^{4}\delta_{i}\mathbf{e}_{i},\ \  \delta_{i}\in\{1,-1\}.$$

Now we still divided $Q$ into $16$ smaller copies of $\frac{1}{2} Q$, that is, $Q=\bigcup_{i=1}^{16}(\frac{1}{2} Q+\mathbf{y}_{i})$ with $\mathbf{y}_{i}\in \{\frac{1}{2}g\mathbf{v}:\mathbf{v}\in\{1,-1\}^{4}\}$. Then we also get 16 corresponding subsets $X_{i}=X\cap (\frac{1}{2} Q+\mathbf{y}_{i})$, $i=1,...,16$.
For every translating point pair $\left(\frac{1}{2}g\mathbf{v}+\mathbf{y}_{i},\frac{1}{2}g(-\mathbf{v})+\mathbf{y}_{i}\right)$, $i=1,...,16$, we will show that at least one point of $\left(\frac{1}{2}g\mathbf{v}+\mathbf{y}_{i},\frac{1}{2}g(-\mathbf{v})+\mathbf{y}_{i}\right)$ lies on the boundary of some facet of $Q$. Without loss of generality, take a point pair $\left(\frac{1}{2}g\mathbf{v}_{0},\frac{1}{2}g(-\mathbf{v}_{0})\right)$ with $\mathbf{v}_{0}=\Sigma_{i=1}^{4}\sigma_{i}\mathbf{e}_{i}$, $\sigma_{i}\in\{1,-1\}$.
Then
\begin{equation}\label{eq3.6}
  \frac{1}{2}g\mathbf{v}_{0}+\mathbf{y}_{i}=\frac{1}{2}g(\Sigma_{i=1}^{4}(\sigma_{i}+\delta_{i})\mathbf{e}_{i}),
\end{equation}
\begin{equation}\label{eq3.7}
\frac{1}{2}g(-\mathbf{v}_{0})+\mathbf{y}_{i}=\frac{1}{2}g(\Sigma_{i=1}^{4}(\delta_{i}-\sigma_{i})\mathbf{e}_{i}).
\end{equation}
If all $\delta_{i}=\sigma_{i}$, $i=1,...,4$, then the point \eqref{eq3.6} is contained in $\cap_{i=1}^{4} F_{\delta_{i}(i)}$; if $\delta_{i}=\sigma_{i}$, $i=1,...,3$ and $\delta_{4}\neq\sigma_{4}$, then the point \eqref{eq3.6} is contained in $\cap_{i=1}^{3} F_{\delta_{i}(i)}\cap Q$; if $\delta_{i}=\sigma_{i}$, $i=1,2$ and $\delta_{j}\neq\sigma_{j}$, $j=3,4$, then the point \eqref{eq3.6} is contained in $\cap_{i=1}^{2} F_{\delta_{i}(i)}\cap Q$; if $\delta_{1}=\sigma_{1}$, $\delta_{i}\neq\sigma_{i}$, $i=2,...,4$, then the point \eqref{eq3.7} is contained in $\cap_{i=2}^{4} F_{\delta_{i}(i)}\cap Q$; if all $\delta_{i}\neq\sigma_{i}$, $i=1,...,4$, then the point \eqref{eq3.7} is contained in $\cap_{i=1}^{4} F_{\delta_{i}(i)}$.

By above discussions and $X$ touches each facet of $Q$ at exactly one interior point, we have $d_{C_{4,p}}(X_{i})<2$ for all $i=1,...,16$. Therefore, $b_{C_{4,p}}(X)\leq 2^{4}$.
\end{enumerate}

\subsection{p=1}
By \eqref{eq0.1}, determining the covering number of a convex body is useful for solving the Borsuk's partition problem. Let $m$ be a positive integer and let $\gamma_{m}(K)$ be the smallest positive number $r$ such that $K$ can be covered by $m$ translates of $rK$. Clearly, $\gamma_{m}(K)<1$ is equivalent to $\gamma(K)\leq m$. Firstly, we estimate the values of $\gamma_{2n}(C_{n,p})$.

\begin{lem}[\cite{Zong}]\label{lem4}
Let $K$ be an $n$-dimensional convex body, $\lambda$ be a real number satisfying $0<\lambda<1$, $R$ be a closed region on $\partial(K)$ with boundary $\Gamma$ and a relatively interior point $\mathbf{p}$. If $\Gamma\cup\{\mathbf{p}\}\subset \lambda K+\mathbf{y}$ holds for some point $\mathbf{y}$, then we have $R\subset \lambda K+\mathbf{y}$.
\end{lem}

\begin{lem}\label{lem5}
$\gamma_{2n}(C_{n,p})\leq(\frac{n-1}{n})^{\frac{1}{p}}$ holds for all $1\leq p\leq2$ and $n\geq2$.
\end{lem}
\begin{proof}
Let $m=(\frac{1}{n})^{\frac{1}{p}}$ and $\lambda=(\frac{n-1}{n})^{\frac{1}{p}}$.
Let $\mathbf{y}_{i}=m\mathbf{e}_{i}$, $\mathbf{y}_{n+i}=-m\mathbf{e}_{i}$, $i=1,...,n$.
Here we show that
\begin{equation}\label{eq3.8}
  C_{n,p} \subseteq \bigcup_{i=1}^{n}((\lambda C_{n,p}+\mathbf{y}_{i})\cup (\lambda C_{n,p}+\mathbf{y}_{n+i})).
\end{equation}

Let $\Pi_{i}=\{\mathbf{x}=(x_{1},x_{2},...,x_{n}):x_{i}=m,\mathbf{x}\in \partial(C_{n,p})\}$. Let $R_{i}$ be a closed region on $\partial(C_{n,p})$ with boundary $\Pi_{i}$ and a relatively interior point $\mathbf{e}_{i}$, $i=1,...,n$. For every point $\mathbf{x}\in\Pi_{i}$, we have
$$|x_{1}|^{p}+|x_{2}|^{p}+...+|x_{i-1}|^{p}+|x_{i+1}|^{p}+...+|x_{n}|^{p}=1-\frac{1}{n}=\frac{n-1}{n},$$
so $\mathbf{x}\in \lambda C_{n,p}+\mathbf{y}_{i}$. And $\mathbf{e}_{i}\in \lambda C_{n,p}+\mathbf{y}_{i}$, since
$$|1-m|^{p}=\left|\left(1-\left(\frac{1}{n}\right)^{\frac{1}{p}}\right)\right|^{p}\leq \frac{n-1}{n}.$$
By Lemma \ref{lem4}, we have $R_{i}\subset \lambda C_{n,p}+\mathbf{y}_{i}$, $i=1,...,n$. In other words, $\{\mathbf{x}:x_{i}\geq m\}\cap C_{n,p}\subseteq \lambda C_{n,p}+\mathbf{y}_{i}$, $i=1,...,n$. Similarly, $\{\mathbf{x}:x_{i}\leq -m\}\cap C_{n,p}\subseteq \lambda C_{n,p}+\mathbf{y}_{n+i}$, $i=1,...,n$.

Since $C_{n,p}\backslash\bigcup_{i=1}^{n}(\{\mathbf{x}:|x_{i}|\geq m\}\cap C_{n,p})$ is an $n$-dimensional parallelotope with vertices $(\pm m,\pm m,...,\pm m)$ and
$$\mathrm{conv}\{\{\mathbf{o}\}\cup \{\mathbf{x}:x_{i}=\pm m,\mid x_{j}\mid\leq m,j\neq i\}\subseteq \mathrm{conv}\{\{\mathbf{o}\}\cup (\{\mathbf{x}:x_{i}=\pm m\}\cap C_{n,p}) \}, i=1,...,n,$$
 it is sufficient to verify that $\mathbf{o}\in \lambda C_{n,p}+\mathbf{y}_{i}$, $i=1,...,2n$.

 As
$$\left(\left(\frac{1}{n}\right)^{\frac{1}{p}}\right)^{p}=\frac{1}{n}\leq \frac{n-1}{n},\ \ n\geq 2,$$
we know $\mathbf{o}\in \lambda C_{n,p}+\mathbf{y}_{i}$, $i=1,...,2n$.

Therefore, \eqref{eq3.8} holds. Consequently,
$$\gamma_{2n}(C_{n,p})\leq\left(\frac{n-1}{n}\right)^{\frac{1}{p}}$$
 holds for all $1\leq p\leq2$ and $n\geq 2$. This completes the proof of the lemma.
\end{proof}

In order to show the case of $p=1$, we use the concept of completeness. A bounded set is called complete if it is not properly contained in a set of the same diameter. Clearly, a complete set is convex and compact. In \cite{Egglestonh}, H. G. Eggleston showed that any bounded set $X\subseteq \mathbb{M}^{n}$ can be embedded in a complete set $A$ of the same diameter, the complete set $A$ is called the \textit{completion }of $X$. Generally, $A$ is not unique. For every bounded set $X\subseteq \mathbb{M}^{n}$, we have $b_{C}(X)\leq b_{C}(A)$ , since $X\subseteq A\cap X\subseteq \cup_{i=1}^{b_{C}(A)}(A_{i}\cap X)=\cup_{i=1}^{b_{C}(A)}(X_{i})$ and $d_{C}(X_{i})=d_{C}(A_{i}\cap X)\leq d_{C}(A_{i})<d_{C}(A)=d_{C}(X)$.

The complete sets have a useful characterization in terms of supporting slabs. A supporting slab of the convex body $K\in\mathcal{K}^{n}$ is any closed set $\Sigma\supseteq K$ that is bounded by two parallel supporting hyperplanes $H$, $H'$ of $K$. The distance between $H$ and $H'$ is called the width of $\Sigma$. For any other convex body $M$, we say that the supporting slab $\Sigma$ of $K$ is $M$-regular if the supporting slab of $M$ that is parallel to $\Sigma$ has the property that at least one of its bounding hyperplanes contains a smooth boundary point of $M$ (a boundary point through which passes only one supporting hyperplane of $M$). In \cite{Moreno} and \cite{Morenoj}, J. P. Moreno and R. Schneider gave a new characterization of the complete sets in $\mathbb{M}^{n}$. They also states that in the case of a polyhedral norm, the space of translation classes of complete sets of given diameter is a finite polytopal complex.
\begin{lem}[\cite{Moreno}]\label{lem2}
Let $d>0$. The $n$-dimensional convex body $K\in \mathcal{K}^{n}$ is a complete set of diameter $d$ if and only if the following properties hold:
\begin{description}
  \item[(a)] Every $C$-regular supporting slab of $K$ has width $\leq d$, $C$ is the unit ball of $\mathbb{M}^{n}$.
  \item[(b)] Every $K$-regular supporting slab of $K$ has width $d$.
\end{description}
\end{lem}

\begin{lem}[\cite{Morenoj}]\label{lem3}
Let $\Sigma_{1}$,...,$\Sigma_{k}$ be the $C$-regular supporting slabs of the polytopal unit ball $C$. Each complete set $K$ with diameter $2$ is of the form
\begin{equation*}
K=\bigcap_{i=1}^{k}(\Sigma_{i}+\mathbf{t}_{i})\label{eq7}
\end{equation*}
with $\mathbf{t}_{i}\in \mathbb{R}^{n}$, $i=1,...,k$.
\end{lem}
For the polytopal unit ball $C_{4,1}$ of $M_{1}^{4}$, its supporting slabs are $\Sigma_{1}$ with outer normal vectors $\pm\mathbf{u}_{1}=\pm(1,1,1,1)$, $\Sigma_{2}$ with outer normal vectors $\pm\mathbf{u}_{2}=\pm(-1,-1,1,1)$, $\Sigma_{3}$ with outer normal vectors $\pm\mathbf{u}_{3}=\pm(1,-1,-1,1)$, $\Sigma_{4}$ with outer normal vectors $\pm\mathbf{u}_{4}=\pm(-1,1,-1,1)$, $\Sigma_{5}$ with outer normal vectors $\pm\mathbf{u}_{5}=\pm(1,1,1,-1)$, $\Sigma_{6}$ with outer normal vectors $\pm\mathbf{u}_{6}=\pm(-1,1,1,1)$, $\Sigma_{7}$ with outer normal vectors $\pm\mathbf{u}_{7}=\pm(1,-1,1,1)$ and $\Sigma_{8}$ with outer normal vectors $\pm\mathbf{u}_{8}=\pm(1,1,-1,1)$.
Each slab $\Sigma_{i}$ is bounded by two parallel hyperplanes $\Phi_{i}=\{\mathbf{x}:\langle \mathbf{x}, \mathbf{u}_{i}\rangle=1\}$ and $\Phi_{-i}=\{\mathbf{x}:\langle \mathbf{x}, -\mathbf{u}_{i}\rangle=1\}$, $i=1,...,8$.

For every bounded set $X\subseteq M_{p}^{4}$ with $d_{C_{4,1}}(X)=2$, there always exists a completion $D$ of $X$. Up to some translation and by Lemma \ref{lem3}, we may assume that
\begin{align*}
   X\subseteq D
    &=D(\alpha_{1},\alpha_{2},\alpha_{3},\alpha_{4})\\
    &=\bigcap_{i=5}^{8}\Sigma_{i}\cap \bigcap_{i=1}^{4}(\Sigma_{i}+\alpha_{i}\mathbf{u}_{i})\\
&=\left(C_{4,1}\cup (\cup _{i=1}^{4} S_{\pm i})\right)\cap \bigcap_{i=1}^{4}(\Sigma_{i}+\alpha_{i}\mathbf{u}_{i})\\
&=\bigcup_{i=5}D_{i},
\end{align*}
where $D_{1}= C_{4,1}\cap \bigcap_{i=1}^{4}(\Sigma_{i}+\alpha_{i}\mathbf{u}_{i})$, $D_{j}=S_{j-1}\cap \bigcap_{i=1}^{4}(\Sigma_{i}+\alpha_{i}\mathbf{u}_{i})$ or $D_{j}=S_{-(j-1)}\cap \bigcap_{i=1}^{4}(\Sigma_{i}+\alpha_{i}\mathbf{u}_{i})$
with $|\alpha_{i}|\leq \frac{1}{4}$, $j=2,...,5$, and
$$S_{ i}=\mathrm{conv}\left((C_{4,1}\cap\Phi_{i})\cup\frac{1}{2}\mathbf{u}_{i}\right), S_{ -i}=-S_{ i},\ \ i=1,...,4.$$

  For $i=1,...,4$, we obtain that $d_{C_{4,1}}(S_{i})=d_{C_{4,1}}(S_{-i})=2$ and that there are exactly five vertices of $S_{ i}$ or $S_{ -i}$ such that the distance between each pair is $2$. Neither $S_{i}$ nor $S_{-i}$ is a complete set by Lemma \ref{lem2}, since there exist a $S_{i}$ $(S_{-i})$-regular supporting slab of $S_{i}$ $(S_{-i})$ with width $1$. By Lemma \ref{lem5}, $C_{4,1}$ can be covered by $8$ smaller copies of $C_{4,1}$. In fact, some vertices with neighbour also have been covered from the covering of $C_{4,1}$, so the remaining part of $S_{ i}$ or $S_{ -i}$ has diameter strictly smaller than $2$. Therefore, $X$ can be divided into at most $12$ parts, each of which has diameter strictly smaller than $2$. Consequently, $b_{C_{4,1}}(X)\leq 2^{4}$.

In conclusion, $b_{C_{4,p}}(X)\leq 2^4$ holds for all four-dimensional $\ell_{p}$ space and all bounded set $X$. This completes the proof of the theorem.

\section*{Acknowledgements}
We are very grateful to professor Chuanming Zong for his supervision and discussion. This work is supported by the National Natural Science Foundation of China (NSFC11921001), the National Key Research and Development Program of China (2018YFA0704701) and the Natural Science Foundation of Jiangsu Province (BK20210555).

\end{document}